\documentclass[12pt]{article}
\usepackage{mathrsfs}
\usepackage{epic,eepic,epsf,epsfig}
\usepackage{amsfonts,srcltx,mathrsfs}
\usepackage{multirow}
\usepackage{amsmath}
\usepackage{amssymb}
\usepackage{amsbsy}
\usepackage{graphicx}
\usepackage{amsfonts}
\usepackage{color}
\usepackage{setspace}

\newtheorem{lem}{Lemma}[section]%
\newtheorem{theorem}[lem]{Theorem}%

\def\nd{\mathrel{\bigm|\kern-.7em/}}

\def\f{\noindent}

\def\P\GammaL{\hbox{\rm P\GammaL}}

\begin{document}
\title{A lower bound of toughness of regular graphs: in terms of second largest eigenvalue}

\footnotetext{E-mails: zhangwq@pku.edu.cn}

\author{Wenqian Zhang\\
{\small School of Mathematics and Statistics, Shandong University of Technology}\\
{\small Zibo, Shandong 255000, P.R. China}}
\date{}
\maketitle

\begin{abstract}
Let $G$ be a connected (non-complete) $d$-regular graph with  $d\geq3$. Let $c(G-S)$ denote the number of components of $G-S$ for any cut $S$ of $G$. The toughness $t(G)$ of $G$ is defined as $\min\left\{\frac{|S|}{c(G-S)}\right\}$, where the minimum is taken over all proper cuts $S$ of $G$. Let $\lambda_{2}(G)$ denote the second largest eigenvalue of $G$. In this paper, we prove
$$t(G)\geq\min\left\{\frac{d+1}{d}(d-\lambda_{2}(G)),1\right\}.$$
\bigskip

\f {\bf Keywords:} eigenvalue; second largest eigenvalue; toughness; regular graph.

\end{abstract}

 \baselineskip 17 pt

\section{Introduction}

In this paper, all graphs considered are finite, undirected and simple.  For a graph $G$, the vertex set and edge set of $G$ are denoted by $V(G)$ and $E(G)$, respectively. Let $e(G)=|E(G)|$.  For two disjoint subsets $X$ and $Y$ of $V(G)$, let $e_{G}(X,Y)$ be the number of edges between  $X$ and  $Y$ in $G$. For a subset $S$ of $V(G)$, let $G[S]$ denote the subgraph induced by $S$, and let $G-S=G[V(G)-S]$.  For an integer $n\geq1$, let $K_{n}$ denote the complete graph of order $n$. For any terminology used but not defined here, one may refer to \cite{BH}.

 Let $G$ be a graph with vertices $u_{1},u_{2},\ldots,u_{n}$. The {\em adjacency matrix} $A(G)$ of $G$ is an $n\times n$ square matrix $(a_{ij})$, where $a_{ij}=1$ if $u_{i}$ is adjacent to $u_{j}$, and $a_{ij}=0$ otherwise. The eigenvalues of $G$ are the eigenvalues of its adjacency matrix $A(G)$, which are denoted by $\lambda_{1}(G)\geq \lambda_{2}(G)\geq\cdots\geq \lambda_{n}(G)$. As is well known, if $G$ is $d$-regular, then $\lambda_{1}(G)=d$. Moreover, $\lambda_{2}(G)<d$ if $G$ is also connected.

Let $G$ be a connected (and non-complete) graph of order $n$.  A {\em cut} of $G$ is a set $S\subseteq V(G)$ such that $G-S$ is disconnected.  Let $c(G-S)$ denote the number of components of $G-S$ for any cut $S$. Chv\'{a}tal \cite{C} introduced the concept of {\em toughness} $t(G)$ of $G$. It is defined as $t(G)=\min\left\{\frac{|S|}{c(G-S)}\right\}$, where the minimum is taken over all proper cuts $S$ of $G$. Toughness is a measure of the connectivity of graphs. 
Now assume that $G$ is $d$-regular. The study of relationship between $t(G)$ and eigenvalues of $G$ has been a hot topic. Set $\lambda=\max\left\{\lambda_{2}(G),-\lambda_{n}(G)\right\}$. Alon \cite{A} proved that $t(G)>\frac{1}{3}(\frac{d^{2}}{d\lambda+\lambda^{2}}-1)$, by which graphs with large toughness and large girth are constructed. Around the same time, Brouwer \cite{B} independently discovered a slightly better bound $t(G)>\frac{d}{\lambda}-2$. And in \cite{B1}, Brouwer made the following conjecture: 
 $$t(G)\geq\frac{d}{\lambda}-1.$$
 It is mentioned that there are infinitely many graphs $G$ satisfying $t(G)=\frac{d}{\lambda}$ (for example, strongly regular graphs \cite{B,CW}, and many Kneser graphs \cite{POH}).
Gu \cite{G0} improved the result of Brouwer \cite{B}, and settled this conjecture in \cite{G}.
Haemers \cite{H} proposed a stronger conjecture:
$$t(G)\geq\frac{d-\lambda_{2}(G)}{-\lambda_{n}(G)}.$$
There are very few progress on Haemers' conjecture \cite{H}.
 
 For the special case of toughness 1, Liu and Chen \cite{LC} improved the previous results as follows.

\begin{theorem}{\rm (Liu and Chen \cite{LC})}\label{toughness Liu}
Let $G$ be a connected $d$-regular graph with $d\geq3$. If
\begin{align*}
\begin{split}
\lambda_{2}(G)<\left\{
\begin{array}{lll}
d-1+\frac{3}{d+1} &for ~even~d,\\
\\
d-1+\frac{2}{d+1} &for ~odd~d,
\end{array}
\right.
\end{split}
\end{align*}
then $t(G)\geq1$.
\end{theorem}

Cioab\u{a} and Wong \cite{CW} further improved Theorem \ref{toughness Liu} by obtaining a sharp upper bound of $\lambda_{2}(G)$ (see Theorem \ref{toughness CW}). Later, Cioab\u{a} and Gu \cite{CG2} also established eigenvalue conditions for $G$ with given edge-connectivity to
guarantee that $t(G)\geq1$.

\begin{theorem}{\rm (Cioab\u{a} and Wong \cite{CW})}\label{toughness CW}
Let $G$ be a connected $d$-regular graph with $d\geq3$. If
\begin{align*}
\begin{split}
\lambda_{2}(G)<\left\{
\begin{array}{lll}
\frac{d-2+\sqrt{d^{2}+12}}{2} &for ~even~d,\\
\\
\frac{d-2+\sqrt{d^{2}+8}}{2} &for ~odd~d,
\end{array}
\right.
\end{split}
\end{align*}
then $t(G)\geq1$.
\end{theorem}

Clearly, $t(G)\leq1$ for any  $d$-regular bipartite graph $G$. Zhang \cite{Z} obtained a sharp upper bound of $\lambda_{2}(G)$ to guarantee  $t(G)>1$ when $G$ is a non-bipartite $d$-regular graph.

\begin{theorem}{\rm (Zhang \cite{Z})}\label{main2}
Let $G$ be a connected non-bipartite $d$-regular graph with $d\geq3$. If $\lambda_{2}(G)<\frac{\sqrt{1+4(d-1)^{2}}-1}{2}$, then $t(G)>1$.
\end{theorem}

Chen, Lin and Wang \cite{CLW} gave a sharp upper bound of $\lambda_{2}(G)$ to guarantee  $t(G)\geq\frac{1}{b}$, where $b\geq2$  is an integer.

\begin{theorem}{\rm (Chen, Lin and Wang \cite{CLW})}\label{1/b}
Let $G$ be a connected $d$-regular graph with integers $d\leq b^{2}+b$ and $b\geq2$. If
\begin{equation}
\lambda_{2}(G)<\left\{
\begin{array}{ll}
\frac{d-2+\sqrt{d^{2}+4(d-b)+8}}{2} & for~odd~b,\\
\\
\frac{d-2+\sqrt{d^{2}+4(d-b)+4}}{2}& for~even~b,
\end{array}
\right.\notag
\end{equation}
then $t(G)\geq\frac{1}{b}$.
\end{theorem}

In this paper, we gave a lower bound of $t(G)$ in terms of $\lambda_{2}(G)$.

\begin{theorem}\label{main}
Let $G$ be a connected $d$-regular graph with $d\geq3$. Then $$t(G)\geq\min\left\{\frac{d+1}{d}(d-\lambda_{2}(G)),1\right\}.$$
\end{theorem}

Clearly, Theorem \ref{main} improved Haemers' conjecture \cite{H} for regular graphs $G$
 with $t(G)<1$. 
 
We point out that the bound in Theorem \ref{main} is essentially tight for regular graphs $G$ with $t(G)<1$. For example, let $k$ be an even integer with $2\leq k<d$, and let $K^{-k}_{d+1}$ denote the graph obtained from $K_{d+1}$ by deleting a matching with $\frac{k}{2}$ edges. Let $H_{d}$ be a $d$-regular graph of order $k+d(d+1)$ obtained from a set $S$ of $k$ independent vertices  and $d$ copies of $K^{-k}_{d+1}$, by adding a matching with $k$ edges between $S$ and the $k$ vertices of degree $d-1$ in each copy of  $K^{-k}_{d+1}$. Clearly, $t(H_{d})\leq\frac{k}{d}$ for considering the cut $S$. Similar to Lemma 8 of \cite{CGH} or 
Proposition 2 of \cite{CLW}, we can show that 
$$\lambda_{2}(H_{d})=\frac{d-2+\sqrt{d^{2}+4(d-k)+4}}{2}.$$
Thus 
$$d-\lambda_{2}(H_{d})=\frac{2k}{d+2+\sqrt{d^{2}+4(d-k)+4}}>\frac{k}{d+2}.$$
However, Theorem \ref{main} implies $t(G)>\frac{k}{d}$ if $d-\lambda_{2}(G)>\frac{k}{d+1}$, which is close to $\frac{k}{d+2}$. 

\medskip

The rest of the paper is organized as follows. In Section 2, we will prove a useful lemma on the largest root of a special kind of polynomials. In Section 3, we will give the proof of Theorem \ref{main}.

\section{A useful lemma}

Let $\widehat{y}$ represent the vanishing of $y$.
For example, 
$$y_{1}\cdots \widehat{y_{i}}\cdots y_{n}=\prod_{1\leq j\leq n,j\neq i}y_{j}.$$

\medskip

\f{\bf Definition 1.}
Let $a_{1}\geq a_{2}\geq\cdots\geq a_{n}\geq0$ and $b_{1}, b_{2},..., b_{n}>0$ be real numbers.
Define a polynomial 
$$f^{b_{1}, b_{2},..., b_{n}}_{a_{1}, a_{2},..., a_{n}}(x)=(x-a_{1})(x-a_{2})\cdots(x-a_{n})+\sum_{1\leq i\leq n}b_{i}(x-a_{1})\cdots\widehat{(x-a_{i})}\cdots(x-a_{n}).$$

\medskip

\begin{lem}\label{lem1}
For numbers $a_{1}\geq a_{2}\geq\cdots\geq a_{n}\geq0$ and $b_{1}, b_{2},..., b_{n}>0$, let
 $f^{b_{1}, b_{2},..., b_{n}}_{a_{1}, a_{2},..., a_{n}}(x)$ be defined as above. Then $f^{b_{1}, b_{2},..., b_{n}}_{a_{1}, a_{2},..., a_{n}}(x)$ has $n$ real roots.
\end{lem}

\f{\bf Proof:} We prove the lemma in the following two cases. 

\medskip

\f{\bf Case 1.}  $a_{1}, a_{2},..., a_{n}$ are distinct, i.e., $a_{1}> a_{2}>\cdots> a_{n}\geq0$. 

\medskip

Clearly, $$f^{b_{1}, b_{2},..., b_{n}}_{a_{1}, a_{2},..., a_{n}}(a_{1})=b_{1}(a_{1}-a_{2})\cdots(a_{1}-a_{n})>0,$$
 and similarly
$$f^{b_{1}, b_{2},..., b_{n}}_{a_{1}, a_{2},..., a_{n}}(a_{2})<0,$$
$$......$$
$$f^{b_{1}, b_{2},..., b_{n}}_{a_{1}, a_{2},..., a_{n}}(a_{n-1})(-1)^{n}>0,$$
and
$$f^{b_{1}, b_{2},..., b_{n}}_{a_{1}, a_{2},..., a_{n}}(a_{n})(-1)^{n+1}\leq0.$$ 
By Rolle Theorem, we can find $n-1$ real roots respectively between $a_{i}$ and $a_{i+1}$ for all $1\leq i\leq n-1$. Hence, $f^{b_{1}, b_{2},..., b_{n}}_{a_{1}, a_{2},..., a_{n}}(x)$ must have $n$ real roots.

\medskip

\f{\bf Case 2.} There are exactly $r$ distinct numbers among $a_{1}, a_{2},..., a_{n}$, say $c_{1}> c_{2}>\cdots> c_{r}\geq0$.

\medskip

 Let $m_{i}$ denote the multiplicity of $c_{i}$ for any $1\leq i\leq r$. Then $m_{1}+m_{2}+\cdots+m_{r}=n$. Moreover,
 $$f^{b_{1}, b_{2},..., b_{n}}_{a_{1}, a_{2},..., a_{n}}(x)=(x-c_{1})^{m_{1}-1}(x-c_{2})^{m_{2}-1}\cdots (x-c_{r})^{m_{r}-1}f^{b'_{1}, b'_{2},..., b'_{r}}_{c_{1}, c_{2},..., c_{r}}(x),$$
  where
  $$b'_{i}=\sum^{m_{1}+m_{2}+\cdots+m_{i}}_{j=m_{1}+m_{2}+\cdots+m_{i-1}+1}b_{j}>0$$
   for any $1\leq i\leq r$. 
By the result of Case 1, we know that $f^{b'_{1}, b'_{2},..., b'_{r}}_{c_{1}, c_{2},..., c_{r}}(x)$ has $r$ real roots. Thus, $f^{b_{1}, b_{2},..., b_{n}}_{a_{1}, a_{2},..., a_{n}}(x)$ has $n$ real roots. \hfill$\Box$

\begin{lem}\label{lem2}
For $a_{1}\geq a_{2}\geq\cdots\geq a_{n}\geq0$ and $b_{1}, b_{2},..., b_{n}>0$,
let $f^{b_{1}, b_{2},..., b_{n}}_{a_{1}, a_{2},..., a_{n}}(x)$ be defined as above. Let $\rho$ be the largest root of $f^{b_{1}, b_{2},..., b_{n}}_{a_{1}, a_{2},..., a_{n}}(x)$. Then $a_{1}\geq\rho\geq\frac{b_{i}}{b_{1}+b_{i}}a_{1}+\frac{b_{1}}{b_{1}+b_{i}}a_{i}$ for any $2\leq i\leq n$.
\end{lem}

\f{\bf Proof:} Clearly, $f^{b_{1}, b_{2},..., b_{n}}_{a_{1}, a_{2},..., a_{n}}(x)>0$ for any $x>a_{1}$. This implies that $\rho\leq a_{1}$. Since
$$f^{b_{1}, b_{2},..., b_{n}}_{a_{1}, a_{2},..., a_{n}}(a_{2})=b_{2}(a_{2}-a_{1})(a_{2}-a_{3})\cdots(a_{2}-a_{n})\leq0,$$
we have $\rho\geq a_{2}$ by Rolle Theorem.

Let  $2\leq i\leq n$ be fixed. Set $a=\frac{b_{i}}{b_{1}+b_{i}}a_{1}+\frac{b_{1}}{b_{1}+b_{i}}a_{i}$. Clearly, $a\leq a_{1}$. It suffices to show $\rho\geq a$. Suppose that $\rho< a$.   Then 
$$f^{b_{1}, b_{2},..., b_{n}}_{a_{1}, a_{2},..., a_{n}}(a)>0.$$
 Since $\rho\geq a_{2}$ and $a>\rho$, we have $a>a_{j}$ for any $2\leq j\leq n$. Recall $a\leq a_{1}$. Clearly, 
$$f^{b_{1}, b_{2},..., b_{n}}_{a_{1}, a_{2},..., a_{n}}(x)=(x-a_{1})(x-a_{i})f^{b_{2},b_{3},..., \widehat{b_{i}},..., b_{n}}_{a_{2}, a_{3},...,\widehat{a_{i}},..., a_{n}}(x)+$$
$$\left(b_{1}(x-a_{i})+b_{i}(x-a_{1})\right)(x-a_{2})\cdots\widehat{(x-a_{i})}\cdots(x-a_{n}).$$
Since $a_{1}\geq a>a_{j}$ for any $2\leq j\leq n$, we have 
$$f^{b_{2},b_{3},..., \widehat{b_{i}},..., b_{n}}_{a_{2}, a_{3},...,\widehat{a_{i}},..., a_{n}}(a)\geq0.$$
 Note that $b_{1}(a-a_{i})+b_{i}(a-a_{1})=0$ as $a=\frac{b_{i}}{b_{1}+b_{i}}a_{1}+\frac{b_{1}}{b_{1}+b_{i}}a_{i}$. It follows that 
$$f^{b_{1}, b_{2},..., b_{n}}_{a_{1}, a_{2},..., a_{n}}(a)\leq0,$$
 a contradiction as $f^{b_{1}, b_{2},..., b_{n}}_{a_{1}, a_{2},..., a_{n}}(a)>0$. 
 Hence, we must have $\rho\geq a$. \hfill$\Box$

\section{Proof of Theorem \ref{main}}

To prove Theorem  \ref{main}, we first introduce the eigenvalue interlacing technique.
Let $G$ be a graph of order $n$. For a partition $\left\{V_{1},V_{2},...,V_{m}\right\}$ of $V(G)$, the {\em quotient matrix} of $G$ corresponding to this partition is an $m\times m$ matrix $(b_{ij})$, where $b_{ij}=\frac{e_{G}(V_{i},V_{j})}{|V_{i}|}$ for any $1\leq i\neq j\leq m$ and $b_{ii}=\frac{2e(G[V_{i}])}{|V_{i}|}$ for $1\leq i\leq m$. The following interlacing theorem is taken from \cite{BH} (see Chapter 3).

\begin{theorem} {\rm (Brouwer and  Haemers \cite{BH})} \label{interlacing}
Let $G$ be a graph of order $n$, and let $B$ be the quotient matrix of $G$ corresponding to partition $\left\{V_{1},V_{2},...,V_{m}\right\}$ of $V(G)$. Then $\lambda_{i+n-m}(G)\leq \lambda_{i}(B)\leq\lambda_{i}(G)$ for any $1\leq i\leq m$. 
\end{theorem}

Now we are ready to prove Theorem \ref{main}.

\medskip

\f{\bf Proof of Theorem \ref{main}.} Denote by $t=t(G)$. We can assume $t<1$, otherwise there is nothing to prove. It suffices to prove $t\geq\frac{d+1}{d}(d-\lambda_{2}(G))$. By definition of $t$, there is a cut $S\neq\emptyset$ of $G$ such that $c(G-S)=\frac{|S|}{t}$. Set $c=c(G-S)$ and $s=|S|$. Then $s=tc<c$. 

Let $Q_{1},Q_{2},...,Q_{c}$ be the components of $G-S$. Then $V(G)=S\cup(\cup_{1\leq i\leq c}V(Q_{i}))$. For any $1\leq i\leq c$, set $|V(Q_{i})|=n_{i}$ and $e_{G}(S,V(Q_{i}))=e_{i}\geq1$. Now partition $V(G)$ into $c+1$ parts: $S, V(Q_{1}),v(Q_{2}),...,V(Q_{c})$. 
The quotient matrix $B$ of $G$ corresponding to this partition is equal to
\begin{center}
$\left(\begin{array}{ccccc}
 d-\frac{\sum_{1\leq i\leq c}e_{i}}{s}&\frac{e_{1}}{s}&\frac{e_{2}}{s}&\cdots&\frac{e_{c}}{s}\\
 \frac{e_{1}}{n_{1}}&d-\frac{e_{1}}{n_{1}}&0&\cdots&0\\
\frac{e_{2}}{n_{2}}&0&d-\frac{e_{2}}{n_{2}}&\cdots&0\\
\cdots&\cdots&\cdots&\cdots&\cdots\\
\frac{e_{c}}{n_{c}}&0&0&\cdots&d-\frac{e_{c}}{n_{c}}\\
\end{array}\right)$.
   \end{center}
For $1\leq i\leq c+1$, let $\lambda_{i}(B)$ denote the $i$-th largest eigenvalue of $B$. By Theorem \ref{interlacing}, we have 
$$\lambda_{2}(G)\geq\lambda_{2}(B).$$
By a  calculation, the characteristic polynomial $g(B,x)$ of $B$ satisfies 
$$g(B,x)=(x-d)\cdot h(B,x),$$ where
 $$h(B,x)=(x-d+\frac{e_{1}}{n_{1}})(x-d+\frac{e_{2}}{n_{2}})\cdots(x-d+\frac{e_{c}}{n_{c}})
+\sum_{1\leq i\leq n}\frac{e_{i}}{s}(x-d+\frac{e_{1}}{n_{1}})\cdots\widehat{(x-d+\frac{e_{i}}{n_{i}})}
\cdots(x-d+\frac{e_{c}}{n_{c}}).$$ 

Without loss of generality, assume $\frac{e_{1}}{n_{1}}\leq\frac{e_{2}}{n_{2}}\leq\cdots\leq\frac{e_{c}}{n_{c}}$. Note $\frac{e_{i}}{s}>0$ for $1\leq i\leq c$ as $G$ is connected. Clearly, 
$$h(B,x)=f^{\frac{e_{1}}{s}, \frac{e_{2}}{s},..., \frac{e_{c}}{s}}_{d-\frac{e_{1}}{n_{1}}, d-\frac{e_{2}}{n_{2}},..., d-\frac{e_{c}}{n_{c}}}(x).$$
Let $\rho$ denote the largest root of $f^{\frac{e_{1}}{s}, \frac{e_{2}}{s},..., \frac{e_{c}}{s}}_{d-\frac{e_{1}}{n_{1}}, d-\frac{e_{2}}{n_{2}},..., d-\frac{e_{c}}{n_{c}}}(x)$. Clearly, $\rho=\lambda_{2}(B)$. Thus, $\lambda_{2}(G)\geq\rho$. 
By Lemma \ref{lem2}, we have $$\rho\geq\frac{\frac{e_{i}}{s}}{\frac{e_{1}}{s}+\frac{e_{i}}{s}}(d-\frac{e_{1}}{n_{1}})
+\frac{\frac{e_{1}}{s}}{\frac{e_{1}}{s}+\frac{e_{i}}{s}}(d-\frac{e_{i}}{n_{i}})
=d-\frac{e_{1}\cdot e_{i}}{e_{1}+e_{i}}(\frac{1}{n_{1}}+\frac{1}{n_{i}}).$$
 for any $2\leq i\leq c$.
 Thus,
  $$\lambda_{2}(G)\geq d-\frac{e_{1}\cdot e_{i}}{e_{1}+e_{i}}(\frac{1}{n_{1}}+\frac{1}{n_{i}})$$
 for any $2\leq i\leq c$.

Recall that $s\leq c-1$. Now we show that there are at least two components $Q$ of $G-S$ such that $e_{G}(V(Q),S)<d$. Otherwise, we have 
$$sd\geq e_{G}(S,V(G)-S)>d(c-1)\geq sd,$$
a contradiction.
Hence, there are at least two components $G-S$, say $Q'_{1}$ and $Q'_{2}$, such that $e_{G}(V(Q'_{i}),S)<d$ for $i=1,2$. Without loss of generality, assume that $Q'_{1}$ is a  component of $G-S$ with $e_{G}(V(Q'_{1}),S)$ minimum among all the $c$ components of $G-S$. And $Q'_{2}$ is a  component of $G-S$ with $e_{G}(V(Q'_{2}),S)$ minimum among other $c-1$ components (except $Q'_{1}$) of $G-S$. Let $e'_{i}=e_{G}(V(Q'_{i}),S)$ for $i=1,2$. It follows that
$$e'_{1}\leq\frac{e'_{1}+e'_{2}}{2}\leq\frac{e_{G}(S,V(G)-S)}{c}\leq\frac{sd}{c}=td.$$
Let $n'_{i}=|Q'_{i}|$ for $i=1,2$. We claim that $n'_{i}\geq d+1$, otherwise 
$$e_{G}(V(Q'_{i}),S)\geq(d+1-n'_{i})n'_{i}\geq d,$$
a contradiction. Now we have the following two cases.

\medskip

\f{\bf Case 1.} There is another component of $G-S$ (except $Q'_{1}$), say $Q'$, such that $\frac{e_{G}(S,Q')}{|Q'|}\leq\frac{e'_{1}}{n'_{1}}$.

\medskip

In this case, we have $$\frac{2e(Q'_{1})}{n'_{1}}=d-\frac{e'_{1}}{n'_{1}}\geq d-\frac{td}{d+1}$$
and
$$\frac{2e(Q')}{|Q'|}=d-\frac{e_{G}(S,Q')}{|Q'|}\geq d-\frac{e'_{1}}{n'_{1}}\geq d-\frac{td}{d+1}.$$
By eigenvalue interlacing (see \cite{BH}), we have
$$\lambda_{2}(G)\geq\min\left\{\frac{2e(Q'_{1})}{n'_{1}},\frac{2e(Q')}{|Q'|}\right\}\geq
d-\frac{td}{d+1}$$
It follows that $t\geq\frac{d+1}{d}(d-\lambda_{2}(G)),$ as desired.

\medskip

\f{\bf Case 2.} For any component $Q\neq Q'_{1}$ of $G-S$, it holds that $\frac{e_{G}(S,Q)}{|Q|}>\frac{e'_{1}}{n'_{1}}$.

\medskip

Recall that $\frac{e_{1}}{n_{1}}\leq\frac{e_{2}}{n_{2}}\leq\cdots\leq\frac{e_{c}}{n_{c}}$ for all the components $Q_{1},Q_{2},...,Q_{c}$ of $G-S$. In this case, we must have $Q'_{1}=Q_{1}$. There is a $2\leq i_{0}\leq c$ such that $Q'_{2}=Q_{i_{0}}$. Then  (see above) $$\frac{e_{1}+e_{i_{0}}}{2}=\frac{e'_{1}+e'_{2}}{2}\leq td.$$ 
Moreover, $n_{1}=n'_{1}\geq d+1$ and $n_{i_{0}}=n'_{2}\geq d+1$.

Recall that $\lambda_{2}(G)\geq d-\frac{e_{1}\cdot e_{i}}{e_{1}+e_{i}}(\frac{1}{n_{1}}+\frac{1}{n_{i}})$
 for any $2\leq i\leq c$. In particular, 
 $$\lambda_{2}(G)\geq d-\frac{e_{1}\cdot e_{i_{0}}}{e_{1}+e_{i_{0}}}(\frac{1}{n_{1}}+\frac{1}{n_{i_{0}}})\geq
 d-\frac{e_{1}+e_{i_{0}}}{4}(\frac{1}{n_{1}}+\frac{1}{n_{i_{0}}})\geq d-\frac{td}{d+1}.$$
It follows that $t\geq\frac{d+1}{d}(d-\lambda_{2}(G)),$ as desired.
This completes the proof.
\hfill$\Box$

\medskip

\f{\bf Data availability statement}

\medskip

There is no associated data.

\medskip

\f{\bf Declaration of Interest Statement}

\medskip

There is no conflict of interest.

\end{document}